\documentclass{amsart}

\usepackage{amsmath, amssymb, amsthm}
\usepackage{cite}
\usepackage{graphicx}
\usepackage{pinlabel}
\usepackage{enumitem}
\usepackage{verbatim}
\usepackage{epigraph}
\usepackage{color}

\newcommand{\SL}[2][\mathbb{Z}]{\mathrm{SL}_{#2}(#1)}
\newcommand{\GL}[2][\mathbb{Z}]{\mathrm{GL}_{#2}(#1)}
\newcommand{\norm}[1]{\lVert#1\rVert}

\newcommand{\SLdZ}{\mathrm{SL}_{d} (\mathbb{Z})}

\newcommand{\SLOK}{\mathrm{SL}_{2} (\mathcal{O}_K)}

\newcommand{\R}{\mathbb{R}}
\newcommand{\Z}{\mathbb{Z}}
\newcommand{\N}{\mathbb{N}}

\newcommand{\Q}{\mathbb{Q}}

\newcommand{\CAT}[1]{\mathrm{CAT}(#1)}

\newcommand{\Id}{\mathrm{Id}}
\newcommand{\modulus}[1]{\left| #1 \right|}

\newcommand{\BS}[1]{\mathrm{BS}(1,#1)}

\newcommand{\CLF}{\mathrm{CLF}}

\newcommand{\Aut}{\mathrm{Aut}}
\newcommand{\TCLF}{{\mathrm{TCL}}}
\newcommand{\RCLF}{\mathrm{RCL}}
\newcommand{\transpose}{^{\mathrm{T}}}

\newcommand{\ldist}{\mathfrak{d}}
\newcommand{\invdist}{\overline{\mathfrak{d}}}

\newtheorem{thm}{Theorem}[section]
\newtheorem{cor}[thm]{Corollary}
\newtheorem{lemma}[thm]{Lemma}
\newtheorem{prop}[thm]{Proposition}

\newtheorem{thmspecial}{Theorem}

\theoremstyle{definition}
\newtheorem{example}[thm]{Example}
\newtheorem*{defn}{Definition}

\newenvironment{remark}[1][Remark]{\begin{trivlist}
\item[\hskip \labelsep {\bf #1}:]}{\end{trivlist}}

\newenvironment{question}[1][Question]{\begin{trivlist}
\item[\hskip \labelsep {\bf #1}:]}{\end{trivlist}}

\title{Conjugacy length in group extensions}
\author{Andrew Sale} %\address{Balliol College, University of Oxford, United Kingdom}
\thanks{The author was supported by the EPSRC}

\begin{document}

\begin{abstract}
Determining the length of short conjugators in a group can be considered as an effective version of the conjugacy problem. The conjugacy length function provides a measure for these lengths. We study the behaviour of conjugacy length functions under group extensions, introducing the twisted and restricted conjugacy length functions. We apply these results to show that certain abelian-by-cyclic groups have linear conjugacy length function and certain semidirect products $\Z^d \rtimes \Z^k$ have at most exponential (if $k >1$) or linear (if $k=1$) conjugacy length functions.
\end{abstract}

\maketitle

Following a recent trend in geometric group theory towards producing effective results, we study a geometric question closely associated to the conjugacy problem of Max Dehn \cite{Dehn12}, which could be described as an \emph{effective conjugacy problem}. We ask if one can find an estimate of the length of short conjugators between elements in a group. This is measured by the conjugacy length function.

Suppose word-lengths in a group $\Gamma$, with respect to the given finite generating set $X$, are denoted by $\modulus{\cdot}$. The \emph{conjugacy length function} was introduced by T. Riley and is the minimal function $\CLF_\Gamma : \N \rightarrow \N$ which satisfies the following: if $u$ is conjugate to $v$ in $\Gamma$ and $\modulus{u}+\modulus{v}\leq n$ then there exists a conjugator $\gamma \in \Gamma$ such that $\modulus{\gamma} \leq \CLF_\Gamma(n)$. One can define it more concretely to be the function which sends an integer $n$ to
		\[\max \big\{ \min \{ \modulus{w} : wu=vw \} : \modulus{u} + \modulus{v} \leq n \ \textrm{and $u$ is conjugate to $v$ in $\Gamma$} \big\}.\]
We know various upper bounds for the conjugacy length function in certain classes of groups. For example, Gromov--hyperbolic groups have a linear upper bound; this is demonstrated by Bridson and Haefliger \cite[Ch.III.$\Gamma$ Lemma 2.9]{BH99}. They also show that $\CAT{0}$ groups have an exponential upper bound for conjugacy length \cite[Ch.III.$\Gamma$ Theorem 1.12]{BH99}. In \cite{Sale12} we obtained a cubic upper bound for the conjugacy length function of a free solvable group, and obtained an expression for it in wreath products of groups. Jing Tao \cite{Tao11} showed that mapping class groups enjoy a linear conjugacy length function, expanding on previous work of Masur and Minsky \cite{MM00} for pseudo-Anosov elements. Behrstock and Dru\c{t}u \cite{BD11} had recently also shown this upper bound for purely reducible elements. Work by Crisp, Godelle and Wiest \cite{CGW09} on the complexity of the conjugacy problem in right-angled Artin groups also show that these groups have a linear conjugacy length function.

In 1977 Collins and Miller showed that the solubility of the conjugacy problem does not pass to finite index subgroups or to finite extensions \cite{CM77}. Recent work of Bogopolski, Martino and Ventura investigate certain group extensions and what circumstances are necessary for the solubility of the conjugacy problem to carry through to the extension \cite{BMV10}. The extensions they study require a strong assumption to be placed on the structure of centralisers in the quotient group, limiting the application of their work. However, their result applies in cases where the quotient is, for example, cyclic (or indeed finite), enabling them to study such groups as abelian-by-cyclic groups or free-by-cyclic groups.

In this paper we look at the length of short conjugators in extensions similar to those considered in \cite{BMV10}, but instead of their assumption on centralisers in the quotient we place a geometric condition on them. Consider the group extension
\begin{equation*}
1 \longrightarrow F \overset{\alpha}{\longrightarrow} G \overset{\beta}{\longrightarrow} H \longrightarrow 1
\end{equation*}
Bogopolski, Martino and Ventura showed that solubility of the conjugacy problem carries through to $G$ provided certain conditions hold. These conditions include:
\begin{itemize}
\item there is an algorithm to determine when two elements in $F$ are conjugate in $G$ (the restricted conjugacy problem);
\item the twisted conjugacy problem is solvable in $F$;
\item the conjugacy problem is solvable in the quotient $H$.
\end{itemize}
We define the restricted and twisted conjugacy length problems in Section \ref{sec:tcl and rcl} where we also define the corresponding analogues of the conjugacy length function for these problems.

As commented above, the results of \cite{BMV10} only apply to group extensions in which the quotient group $H$ satisfies certain conditions on the structure of centralisers of elements in $H$. In particular, the centraliser in $H$ of any $h\in H$ must be virtually cyclic. We can relax this condition slightly, replacing it with a geometric condition, asking for a function $\rho:G \to [0,\infty)$ which, roughly speaking, measures the diameter of the fundamental domain inside $Z_H(\beta(u))$ under the action of its subgroup $\beta(Z_G(u))$ by left-multiplication, for any $u \in G$, where $Z_G(u)$ denotes the centraliser in $G$ of $u$ (see Section \ref{sec:group ext} for a precise definition). If we let 
		\[\rho_n = \max \{ \rho(u) \mid u \in G,\  \modulus{u} \leq n \}\]
then we show the following:

\begin{thmspecial}\label{thmspecial:group ext}
The conjugacy length function of $G$ satisfies
		\[\CLF_G(n) \leq \CLF_H(n) + \max \left\{ \RCLF_F^G(n),\rho_n + \invdist_F^G(\TCLF_F\big(2\delta_F^G(n+\rho_n);A_G^{(n)})\big)\right\}\]
where $\TCLF_F$ is the twisted conjugacy length function of $F$, $\RCLF_F^G$ is the restricted conjugacy length function of $F$ in $G$, which are defined in Section \ref{sec:tcl and rcl}, and $\delta_F^G$ and $\invdist_F^G$ are subgroup distortion functions, defined in Section \ref{sec:distortion}.
\end{thmspecial}

We use Theorem \ref{thmspecial:group ext} to study conjugacy length in certain abelian-by-cyclic groups and abelian-by-abelian groups. Following the work of Bieri and Strebel \cite{BS78}, the finitely presented, torsion-free, abelian-by-cyclic groups are given by presentations of the form
		\[\Gamma_M = \langle t, a_1, \ldots , a_d \mid [a_i,a_j]=1,ta_it^{-1}=\varphi_M(a_i); i,j=1,\ldots,d\rangle\]
where $M=(m_{ij})$ is a $d \times d$ matrix with integer entries and non-zero determinant and $\varphi_M(a_i)=a_1^{m_{1i}}\ldots a_d^{m_{di}}$ for each $i=1,\ldots , d$.

\begin{thmspecial}\label{thmspecial:ab-by-cyc clf}
Suppose $M$ is a diagonalisable matrix, all of whose eigenvalues have absolute value greater than $1$. Then there exists a constant $C$ depending on $M$ such that
		\[\CLF_{\Gamma_M}(n) \leq Cn.\]
\end{thmspecial}

It is worth noting that in particular solvable Baumslag--Solitar groups are covered by Theorem \ref{thmspecial:ab-by-cyc clf} and hence have a linear conjugacy length function. One could also apply Theorem \ref{thmspecial:group ext} to show that lamplighter groups $\Z_q \wr \Z$ have a linear conjugacy length function. This result however is given in \cite{SaleThesis} using different techniques.

Finally we look at a family of abelian-by-abelian groups which are the semidirect product of two free abelian groups.

\begin{thmspecial}\label{thmspecial:SOL}
Let $\Gamma = \Z^d \rtimes_\varphi \Z^k$, where the image of $\varphi : \Z^k \hookrightarrow \SLdZ$ is contained in an $\R$--split torus $T$. Then there exist positive constants $A,B$ such that
\begin{enumerate}[label={({\arabic*})}]
\item if $k=1$ then $\CLF_\Gamma(n) \leq Bn$;
\item if $k>1$ then $\CLF_\Gamma(n) \leq A^n$.
\end{enumerate}
\end{thmspecial} 

This result has consequences for the conjugacy length function of fundamental groups of prime $3$--manifolds. Combining this result with results of Behrstock and Dru\c{t}u \cite{BD11} and Ji, Ogle and Ramsey \cite{JOR10} we get:

\begin{thmspecial}
Let $M$ be a prime $3$--manifold. Then $\CLF_{\pi_1(M)}(n)$ is bounded above by a quadratic function.
\end{thmspecial}

We can also use Theorem \ref{thmspecial:SOL} to say something about conjugacy in Hilbert modular groups. Suppose $\Gamma = \SLOK$, where $\mathcal{O}_K$ is the ring of integers of a totally real field extension $K$ over $\Q$ of degree $d$. Then $\Gamma$ is a lattice in $\SL[\R]{2}^d$ and the intersection of $\Gamma$ with any minimal parabolic subgroup will be isomorphic to $\Z^d \rtimes_\varphi \Z^{d-1}$, one of the groups considered in Theorem \ref{thmspecial:SOL}. We deduce:

\begin{thmspecial}
Let $\Gamma$ be a Hilbert modular group. Then
\begin{enumerate}
\item  there exists a constant $K>0$ such that two elements $u,v\in \Gamma$ which are contained in the same minimal parabolic subgroup of the ambient Lie group are conjugate if and only if there exists a conjugator $\gamma \in \Gamma$ such that
		\[\modulus{\gamma} \leq K^{\modulus{u}+\modulus{v}};\]
\item there exists a constant $L>0$ such that two elements $u,v\in \Gamma$ which are contained in the same unipotent subgroup of the ambient Lie group are conjugate if and only if there exists a conjugator $\gamma \in \Gamma$ such that
		\[\modulus{\gamma} \leq L(\modulus{u}+\modulus{v}).\]
\end{enumerate}
\end{thmspecial}

We begin in Section \ref{sec:distortion} by discussing subgroup distortion. In Section \ref{sec:tcl and rcl} we introduce the twisted and restricted conjugacy length functions, ingredients necessary for Theorem \ref{thmspecial:group ext}, which is the main result of Section \ref{sec:group ext}, concerning the conjugacy length function of group extensions. Theorem \ref{thmspecial:group ext} is applied to abelian-by-cyclic groups in Section \ref{sec:abelian-by-cyclic} and to semidirect products $\Z^d \rtimes \Z^k$ in Section \ref{sec:Z^d rtimes Z^k}.

\subsubsection*{Acknowledgements}The author would like to thank Cornelia Dru\c{t}u for many valuable discussions on material in this paper. Discussions with Romain Tessera were also helpful.

\section{Subgroup distortion}\label{sec:distortion}

Suppose that $G$ is finitely generated with $\modulus{.}$ denoting the word length in $G$ with respect to some finite generating set. It will not always be the case that $F$ is finitely generated. Suppose that $d_F$ is any left-invariant metric on $F$. For example, we may take $d_F$ to be the restriction to $F$ of the word metric on $G$, or, if $F$ is finitely generated, we may take it to be the word metric on $F$ with respect to some finite generating set for it. We will denote by $\modulus{x}_F$ the distance $d_F(e_F,x)$. Let $\delta_F^G:\N\to \N$ be the subgroup distortion function for $(F,d_F)$ in $G$, defined by 
		\[\delta_F^G(n)=\max \{ \modulus{f}_F : f \in F, \modulus{f}\leq n\}.\]
It is well known that the subgroup distortion is independent of generating sets chosen, up to the following equivalence of functions:
given $f,g : \N \to \N$ we will write $f \preceq g$ if there exists a positive constant $C$ such that $f(n) \leq Cg(Cn)$. If $f \preceq g$ and $g \preceq f$ then we will write $f \asymp g$.

We will also need a function which measures the lower bound of the subgroup distortion in $F$. Namely we define a function $\ldist_F^G:\N \to \N$ by
		\[\ldist^G_F (n)=\min \{ \modulus{f}_F : f \in F, \modulus{f}\geq n\}.\]
This satisfies the following relationship for every $a \in F$:
		\begin{equation}\label{eq:distortion}
		\ldist_F^G(\modulus{a}) \leq \modulus{a}_F \leq \delta_F^G(\modulus{a}).
		\end{equation}
Define
\begin{eqnarray*}
\invdist_F^G(n) & := & \min \{k \in \N : \ldist_F^G(k) \geq n \}\\ & = & \min\{k \in \N : \modulus{f}_F \geq n \ \forall f \in F \ \textrm{such that} \ \modulus{f}\geq k \}
\end{eqnarray*}
This is roughly speaking an inverse subgroup distortion function, in the sense that if, for $a \in F$, we know $\modulus{a}$, then $\delta_F^G(\modulus{a})$ gives an upper bound for $\modulus{a}_F$; meanwhile:

\begin{lemma}\label{lem:invdist is inverse}
Let $a \in F$. Then $\modulus{a} \leq \invdist_F^G(\modulus{a}_F).$
\end{lemma}

\begin{proof}
We know by \eqref{eq:distortion} that $\ldist_F^G(\modulus{a}) \leq \modulus{a}_F$. Therefore, by the definition of $\invdist_F^G(\modulus{a}_F)$, and since $\ldist_F^G$ is non-decreasing, $\modulus{a} \leq \invdist_F^G(\modulus{a}_F)$.
\end{proof}

The following says that $\invdist_F^G$ is the minimal function satisfying Lemma \ref{lem:invdist is inverse}.

\begin{lemma}\label{lem:invdist is minimal}
Suppose an non-decreasing function $\tau:\N \to \N$ satisfies $\modulus{a} \leq \tau(\modulus{a}_F)$ for every $a \in F$. Then $\invdist_F^G(n) \leq \tau(n)$ for every $n \in \N$.
\end{lemma}

\begin{proof}
Suppose $a \in F$ and $\modulus{a}\geq \tau(n)$. Then $\tau(\modulus{a}_F)\geq \tau(n)$, hence $\modulus{a}_F \geq n$ since $\tau$ is non-decreasing.
\end{proof}

For emphasis, we repeat the remark that the functions $\delta_F^G$, $\ldist_F^G$ and $\invdist_F^G$ will depend on the metric $d_F$ chosen for $F$. For example, we may take $d_F$ to be given by the word metric on $G$ restricted to $F \times F$. In which case, all three functions will grow linearly.

\begin{example}[Solvable Baumslag--Solitar groups]
\label{example:BS distortion}
Here we consider an example where $\invdist_F^G$ is the inverse of $\delta_F^G$, up the relation $\asymp$ defined above. In general this may not be the case.

For each $m \in \N$, consider the solvable Baumslag--Solitar groups $\BS{m}$given by the presentation $\langle a , b \mid bab^{-1}=a^m\rangle$. This is the semidirect product $\Z[\frac{1}{m}] \rtimes_m \Z$ where the action of $\Z$ on $\Z[\frac{1}{m}]$ is by multiplication by $m$. The subgroup $\Z[\frac{1}{m}]$ corresponds to the subgroup generated by elements $b^{-r}ab^r$, for non-negative integers $r$. We will consider here just the subgroup generated by $a$ and use the word metric given by the generating set $\{ a \}$. One can see that this is exponentially distorted since $a^{m^n}=b^n a b^{-n}$ for any $n \in \N$. To be precise, this relation implies for each $n \in \N$:
		$$\delta_{\langle a \rangle}^{\BS{m}}(2n+1)\geq m^n.$$
In fact, one can show for all $r \in \Z$ (see \cite[Lemma 2.3.3]{SaleThesis}):
\begin{equation}\label{eqn:BS distortion}
\frac{1}{2}\log_m\modulus{r} \leq \modulus{a^r} \leq (m+2)\log_m\modulus{r} + \frac{m}{2}+1.
\end{equation}
Hence we may put $\tau(n)=(m+2)\log_m(n) +\frac{m}{2}+1$ and apply Lemma \ref{lem:invdist is minimal} to get that
		$$\invdist_{\langle a \rangle}^{\BS{m}}(n) \leq (m+2)\log_m(n) +\frac{m}{2}+1.$$
\end{example}

\section{Twisted and restricted conjugacy length functions}\label{sec:tcl and rcl}

In the following, suppose that $\Gamma$ is a group which admits a left-invariant metric $d_\Gamma$. For $\gamma \in \Gamma$, denote by $\modulus{\gamma}_\Gamma$ the distance $d_\Gamma(1,\gamma)$. We will usually omit the subscript in $\modulus{.}_\Gamma$ when we discuss lengths in $\Gamma$, favouring the subscript notation when dealing with subgroups of $\Gamma$.

\subsection*{The twisted conjugacy length function}

We first recall the \emph{twisted conjugacy problem} in a group $\Gamma$. For an automorphism $\varphi$ of $\Gamma$ we say two elements $u,v \in \Gamma$ are $\varphi$--twisted conjugate if there exists $\gamma \in \Gamma$ such that $u = \gamma v \varphi(\gamma)^{-1}$. In such cases we denote this relationship by $u \sim_\varphi v$. The twisted conjugacy problem asks whether there is an algorithm which, on input an automorphism $\varphi$ and two elements $u$ and $v$, determines whether $u \sim_\varphi v$.

Suppose we are given two elements $u$ and $v$ that are $\varphi$--twisted conjugate. We can ask what can be said about the length of the shortest $\gamma$ such that $u=\gamma v \varphi(\gamma)^{-1}$. In particular, we can look for a function $f:\R_+ \to \R_+$ such that whenever $\modulus{u}+\modulus{v}\leq x$, for $x \in \R_+$, then $u \sim_\varphi v$ if and only if there exists $\gamma$ such that $\modulus{\gamma} \leq f(x)$ which satisfies $u=\gamma v \varphi(\gamma)^{-1}$. We call the minimal such function the \emph{$\varphi$--twisted conjugacy length function} and denote it by $\TCLF_\Gamma(n;\varphi)$. Observe that $\CLF_\Gamma(n)=\TCLF_\Gamma(n;\Id)$. We can extend this notation to subsets $A \subseteq \Aut(\Gamma)$, by defining $\TCLF_\Gamma(n;A)=\sup \{ \TCLF_\Gamma(n,\varphi) : \varphi \in A\}$. The \emph{twisted conjugacy length function} of $\Gamma$ is $\TCLF_\Gamma(n)=\TCLF_\Gamma(n;\Aut(\Gamma))$.

\subsection*{The restricted conjugacy length function}

Given a subgroup $B$ of a group $\Gamma$, the \emph{restricted conjugacy problem} of $\Gamma$ to $B$ asks if there is an algorithm which determines when two elements $a,b \in B$ are conjugate in $\Gamma$ (see \cite{BMV10}).

We can associate to the restricted conjugacy length problem a corresponding function, $\RCLF_B^\Gamma:\R_+ \to \R_+$, called the \emph{restricted conjugacy length function} of $B$ from $\Gamma$. It is defined to be the minimal function satisfying the property that whenever $\modulus{a}+\modulus{b}\leq x$, for $a,b \in B$ and $x \in \R_+$, then $a$ is conjugate to $b$ in $\Gamma$ if and only if there exists a conjugator $\gamma \in \Gamma$ for which $\modulus{\gamma} \leq \RCLF_B^\Gamma(x)$.

Note that in the definition of the restricted conjugacy length function we always consider the length of the involved players as elements of $\Gamma$, rather than using a metric $d_B$ on $B$. This naturally leads us to a lower bound for the conjugacy length function of $\Gamma$:
		\[\RCLF_B^\Gamma\leq \CLF_\Gamma.\]
In fact we need not even assume that $B$ is a subgroup to define the restricted conjugacy problem of $B$ from $\Gamma$ and hence $\RCLF_B^\Gamma$. In order for the lower bound above to be useful though, we would need $B$ to be unbounded in $d_\Gamma$.

\begin{example}[Twisted conjugacy length in free abelian groups] \label{example:TCL for abelian}
Let $\Gamma = \Z^r$ for some positive integer $r$. Let $u,v \in \Z^r$ and $\varphi \in \SL{r}$ be diagonalisable with all eigenvalues real and positive. We wish to find some control on the size of the shortest element $\gamma \in \Z^r$ satisfying 
\begin{equation}
\label{eq:abelian}u + \varphi(\gamma) = \gamma + v.
\end{equation}
Suppose $\varphi$ has an eigenvalue equal to $1$ with corresponding eigenspace $E_1$. Let $V$ be the sum of the remaining eigenspaces, so \mbox{$\R^n=E_1 \oplus V$.} With respect to this decomposition, write 
\[\gamma=\gamma_1 + \gamma_2, \ u=u_1+u_2, \ v=v_1+v_2\]
where $\gamma_1,u_1,v_1 \in E_1$ and $\gamma_2,u_2,v_2 \in V$. Equation (\ref{eq:abelian}) tells us that $u_1=v_1$ and $u_2+\varphi'(\gamma_2)=\gamma_2 + v_2$, where $\varphi'$ is a matrix which corresponds to the action of $\varphi$ on $V$ and hence has no eigenvalues equal to $1$. We may therefore take $\gamma_1=0$ and hence assume that $\varphi$ has no eigenvalues equal to $1$.

Rewrite equation (\ref{eq:abelian}) as $(\Id - \varphi)\gamma=u-v$. Since $1$ is not an eigenvalue of $\varphi$, we notice that $\gamma=(\Id - \varphi)^{-1}(u-v)$. Hence
		\[\norm{\gamma} \leq (1+\norm{\varphi}) (\norm{u}+\norm{v}).\]
Therefore, if $\lambda$ is the largest absolute value of an eigenvalue of $\varphi$, then \[\TCLF_{\Z^r}(n;\varphi) \leq (1+\lambda)n.\]
\end{example}

\begin{example}[Restricted conjugacy length in solvable Baumslag--Solitar groups]\label{example:RCL for Baumslag Solitar}
As in Example \ref{example:BS distortion}, let $\BS{m}=\langle a , b \mid bab^{-1}=a^m\rangle$. Suppose $a^r$ is conjugate to $a^s$ in $\BS{m}$, where $r$ and $s$ may be taken to be non-zero. Every element in $\BS{m}$ can be written uniquely in the normal form $b^{-j} a^l b^k$, for some $j,k,l \in \Z$ with $j,k \geq 0$ and, if $j,k$ are both non-zero, then $l$ is not divisible by $m$. Write a conjugator for $a^r$ and $a^s$ in this way. Then $a^r b^{-j} a^l b^k = b^{-j} a^l b^k a^s$, which leads to $b^{-j} a^{rm^j+l}b^k=b^{-j} a^{l+sm^k}b^k$. Note that both sides of this equation are in normal form, since $rm^j+l$ and $l+sm^k$ are divisible by $m$ if and only if $l$ is as well. So $b^{-j} a^l b^k$ is a conjugator if and only if $j,k$ and $l$ satisfy:
		\[rm^j=sm^k.\]
Then we may take $l=0$ and we also have $k-j=\log_m\modulus{r}-\log_m\modulus{s}$. Since $r$ and $s$ are non-zero integers both $\log_m\modulus{r}$ and $\log_m\modulus{s}$ are non-negative. In equation \eqref{eqn:BS distortion} above we noted that $\log_m\modulus{r} \leq 2\modulus{a^r}$, so 
		\[\modulus{b^{k-j}} \leq \modulus{k-j} \leq 2\max\{\modulus{a^r},\modulus{a^s}\}\leq 2(\modulus{a^r}+\modulus{a^s}).\]
This leads to the restricted conjugacy length function
		\[\frac{n-2}{2} \leq \RCLF_{\langle a \rangle}^{\BS{m}}(n) \leq{2n}\]
where the lower bound follows from looking at the conjugate elements $a^{m^r}$ and $a$ and noting that the shortest conjugator for them is $b^r$.
\end{example}

\section{Conjugacy length in group extensions}\label{sec:group ext}

A solution to the conjugacy problem in certain group extensions is given by Bogopolski, Martino and Ventura \cite{BMV10}. Given a short exact sequence
\begin{equation}
\label{eq:ses} 1 \longrightarrow F \overset{\alpha}{\longrightarrow} G \overset{\beta}{\longrightarrow} H \longrightarrow 1
\end{equation}
they show that, under certain conditions, the solubility of the conjugacy problem in $G$ is equivalent to the subgroup $A_G = \{\varphi_g \mid \varphi_g(x)=g^{-1}\alpha(x)g,x\in F, g \in G\}$ of $\Aut(F)$ having solvable orbit problem (that is to say, there is an algorithm which decides whether for any element $u \in F$ there is some $\varphi \in A_G$ such that $u=\varphi(u)$). The conditions that must apply to the short exact sequence are the following:
\begin{enumerate}[label=({\alph*})]
\item\label{item:H solvable CP} $H$ has solvable conjugacy problem;
\item\label{item:F solvable TCP} $F$ has solvable twisted conjugacy problem;
\item\label{item:centraliser condition} for every non-trivial $h \in H$, the subgroup $\langle h \rangle$ has finite index in the centraliser $Z_H(h)$, and one can algorithmically produce a set of coset representatives.
      \newcounter{enumii_saved}
      \setcounter{enumii_saved}{\value{enumii}}
\end{enumerate}
Condition \ref{item:centraliser condition} is rather restrictive. In particular it implies that centralisers in $H$ need to be virtually cyclic. The types of groups which this includes are typically extensions where $H$ is a finitely generated hyperbolic group.

To study conjugacy length in a group extension, it seems natural therefore that we should require an understanding of the conjugacy length in $H$ and the twisted conjugacy length in $F$. We should also expect the restricted conjugacy length function of $G$ to $F$ to make an appearance and there should be some condition based upon the centralisers of elements in $H$.

We will identify $F$ with its image under $\alpha$.  
Let $\modulus{gF}_H=\min\{\modulus{gf} : f \in F\}$ be the quotient metric on $H$. In the following, the twisted conjugacy length function for $F$ is taken with respect to a metric $d_F$ on $F$, and for $f \in F$ we denote $\modulus{f}_F:=d_F(e,f)$.

\begin{thm}\label{thm:group extensions}
Let $G$ be given by the short exact sequence (\ref{eq:ses}). Suppose that it satisfies the following condition:
\begin{enumerate} [label=(c$^\prime$\hspace{-0.3mm})]
\item \label{item:centraliser domain} %there exists a function $\rho : G \to [0,\infty)$ such that for each $u \in G$ the fundamental domain of $\beta(Z_G(u))$ in $Z_H(\beta(u))$ has diameter bounded above by $\rho(u)$.
the function $\rho(u):=\max \{ \modulus{h\beta(Z_G(u))}_H : h \in Z_H(\beta(u))\}$ takes finite values for all $u \in G$.
\end{enumerate}
Then: 
		\[\CLF_G(n) \leq \max \left\{  \RCLF_F^G(n)   , \CLF_H(n) +  \rho_n + \invdist_F^G \left( \TCLF_F\left(2\delta_F^G(n+\rho_n);A_G^{(n)}\right)\right)\right\}\]
where $\rho_n = \max\{\rho(u)\mid u \in G,  \ \modulus{u}\leq n\}$ and $A_G^{(n)}=\{\varphi_u \in A_G \mid u \in G, \ \modulus{u}\leq n \}$.
\end{thm}

\begin{proof}
We split the proof into various cases, according to the relationship between $\beta(u)$ and $\beta(v)$, beginning with the easiest case. Throughout we will make the assumption that $\modulus{u} \leq \modulus{v}$.

\vspace{2mm}
\noindent\underline{\textsc{Case 1:}} $\beta(u)=\beta(v)=e_H$.
\vspace{2mm}

\noindent In this case $u$ and $v$ lie in the image of $\alpha$. We therefore find a conjugator $x \in G$ such that $v = x^{-1}u x$ and $\modulus{x} \leq \RCLF_F^G(\modulus{u}+\modulus{v})$.

\vspace{2mm}
\noindent\underline{\textsc{Case 2:}} $\beta(u)=\beta(v)\neq e_H$.
\vspace{2mm}

\noindent We need to reduce this case to the twisted conjugacy problem in $F$. First though we will determine a coset of $F$ in $G$ in which we will find a conjugator whose length we can estimate. 

Let $\mathcal{H}$ be a set of left-coset representatives of $F$ in $G$ satisfying $\modulus{h} = \modulus{\beta(h)}_H$ for each $h \in \mathcal{H}$. Let $g$ be any conjugator for $u$ and $v$. The following subset of $\mathcal{H}$ gives the representatives for those cosets which contain a conjugator for $u$ and $v$:
		\[\mathcal{H}_{u,v} = \{ h \in \mathcal{H} \mid \exists f \in F \textrm{ such that } h f \in Z_G(u)g \}.\]
Note that the image under $\beta$ of $\mathcal{H}_{u,v}$ will be precisely the image of $Z_G(u)g$: to say $h$ is in $\mathcal{H}_{u,v}$ is equivalent to saying there exists some $f \in F$ such that $hf \in Z_G(u)g$ and since $\beta(h)=\beta(hf)$ for all $f \in F$ we see that $\beta(\mathcal{H}_{u,v}) = \beta(Z_G(u)g)$. 

Choose $h \in \mathcal{H}_{u,v}$ with $\beta(h)$ of minimal size --- we are choosing the coset of $F$ which is closest to the identity among those cosets containing a conjugator. Since $\beta(u)=\beta(v)$ we deduce that $\beta(Z_G(u)g) \subseteq Z_H(\beta(u))$. Hence we may apply condition \ref{item:centraliser domain} and assume $\modulus{h} \leq \rho(u)$.

Since $\beta(h) \in Z_H(\beta(u))$, it follows that $h^{-1}uh=uf_h$ for some $f_h \in F$. Also $\beta(u)=\beta(v)$ implies $u^{-1}v=f \in F$. Let $a \in F$ satisfy the twisted conjugacy relation
\begin{equation}
\label{eq:thm:CLF of extension:twisted equation} 		f=\varphi_u(a)^{-1} f_h a.
\end{equation}
We will first show that $ha$ is a conjugator for $u$ and $v$ and then show we have a control on its size. By unscrambling equation (\ref{eq:thm:CLF of extension:twisted equation}) we obtain the following:
\begin{eqnarray*}
u^{-1}v = f & = & u^{-1}a^{-1}uf_ha \\
& = & u^{-1}a^{-1}h^{-1}uha.
\end{eqnarray*}
Hence $v=(ha)^{-1}u(ha)$ as required.

The size of $a$ is controlled by the twisted conjugacy length function of $F$: 
		\[\modulus{a}_F \leq \TCLF_F(\modulus{f}_F+\modulus{f_h}_F;\varphi_u).\]
Applying the distortion function gives us $\modulus{f}_F \leq \delta_F^G(n)$. Meanwhile $f_h = [u,h]$, so $\modulus{f_h}_F \leq \delta_F^G(2\modulus{u}+2\modulus{h}) \leq \delta_F^G(n+\rho(u))$, since we made the assumption that $\modulus{u}\leq \modulus{v}$. In summary, we have found a conjugator $ha$ satisfying
		\[\modulus{ha} \leq  \rho (u) + \invdist_F^G\Big(\TCLF_F\big(2\delta_F^G(n+\rho(u));\varphi_u\big)\Big)\]
where $\invdist_F^G$ is the inverse subgroup distortion function defined immediately above the statement of the theorem.

\vspace{2mm}
\noindent\underline{\textsc{Case 3:}} $\beta(u)\neq \beta(v)$.
\vspace{2mm}

\noindent Let $u,v$ be conjugate elements in $G$. Then in particular $\beta(u)$ is conjugate to $\beta(v)$ in $H$. Apply the conjugacy length function of $H$ and we get that there exists $h_0 \in H$ such that $\beta(u)=h_0^{-1}\beta(v)h_0$ and 
		\[\modulus{h_0}_H \leq \CLF_H\big(\modulus{\beta(u)}_H+\modulus{\beta(v)}_H\big).\]
Let $g_0$ be a minimal length element in the pre-image $\beta^{-1}(h_0)$. Set $v_0=g_0^{-1}vg_0$. Then $\beta(v_0)=\beta(u)$ and $v_0$ is conjugate to $u$ via an element $g_0$ satisfying
		\[\modulus{g_0} = \modulus{h_0}_H \leq \CLF_H(n).\]
Now we apply Case 2, above, to find a bounded conjugator $ha$ for $u$ and $v_0$. Then all we need to do is to pre-multiply it by $g_0$ to obtain a conjugator for $u$ and $v$. In other words, we have a conjugator $g_0ha$ for $u$ and $v$ such that
		\[\modulus{g_0ha} \leq \CLF_H(n) + \rho (u) + \invdist_F^G\Big(\TCLF_F\big(2\delta_F^G(n+\rho(u));\varphi_u\big)\Big).\]
This is enough to complete the proof.
\end{proof}

By taking a group extension with cyclic quotient we can reduce this to a simpler expression.

\begin{cor}\label{cor:CLF for group extension with quotient Z}
Suppose in the extension given in (\ref{eq:ses}) the quotient $H$ is $\Z$. Then
		\[\CLF_G(n) \leq \max \left\{  \RCLF_F^G(n)   , n + \invdist_F^G\left(\TCLF_F\left(2\delta_F^G(2n);A_G^{(n)}\right)\right)\right\}\]
where $A_G^{(n)}=\{\varphi_u \in A_G \mid u \in G, \ \modulus{u}\leq n \}$.
\end{cor}

\begin{proof}
Let $u,v$ be conjugate in $G$ such that $\modulus{u}+\modulus{v}\leq n$. Since $H=\Z$, the conjugacy length function of $H$ is the zero function. Furthermore we have $\rho(u)\leq \modulus{u}\leq n$.
\end{proof}

A central extension is another situation where the expression is significantly simplified. Unlike with the cyclic extensions, we retain the need to understand the function $\rho$. In particular, if $F$ is contained in the centre of $G$ then Theorem \ref{thm:group extensions} reduces to
		\[\CLF_G(n)\leq \CLF_H(n)+\rho_n.\]
However we can see from this an example of the limitations of this result. If we take the Heisenberg group,
		\[H_3(\Z)=\langle x,y,z \mid [x,y]=z\rangle\]
then this fits into a central extension of the form of (\ref{eq:ses}) with $F=\langle z \rangle$. However, it is not hard to see that the centraliser of $x$ consists precisely of elements of the form $x^rz^s$, for any pair of integers $r,s$. Projecting this centraliser onto $H_3(\Z)/{\langle z \rangle}\cong \Z^2$ gives a copy of $\Z$, implying that $\rho_n$ cannot be finite and Theorem \ref{thm:group extensions} does not apply.

\section{Abelian-by-cyclic groups}\label{sec:abelian-by-cyclic}

An abelian-by-cyclic group $\Gamma$ has a short exact sequence
		\[1 \longrightarrow A \longrightarrow \Gamma \longrightarrow \Z \longrightarrow 1\]
where $A$ is an abelian group. Following the work of Bieri and Strebel \cite{BS78}, the finitely presented, torsion-free, abelian-by-cyclic groups are given by presentations of the form
		\[\Gamma_M = \langle t, a_1, \ldots , a_d \mid [a_i,a_j]=1,ta_it^{-1}=\varphi_M(a_i); i,j=1,\ldots,d\rangle\]
where $M=(m_{ij})$ is a $d \times d$ matrix with integer entries and non-zero determinant and $\varphi_M(a_i)=a_1^{m_{1i}}\ldots a_d^{m_{di}}$ for each $i=1,\ldots , d$. The aim of this section is to give a linear upper bound for the conjugacy length function of a certain family of abelian-by-cyclic groups:

\begin{thm}\label{thm:clf of ab-by-cyc M expanding}
Suppose $M$ is a diagonalisable matrix, all of whose eigenvalues have absolute value greater than $1$. Then there exists a constant $C$ depending on $M$ such that
		\[\CLF_{\Gamma_M}(n) \leq Cn.\]
\end{thm}

The method of proof is to apply Corollary \ref{cor:CLF for group extension with quotient Z}. First we need to understand the distortion inside the abelian subgroup, this is the subject of Lemma \ref{lem:ab-by-cyc distortion}. Following this, we calculate estimates for the relevant restricted and twisted conjugacy length functions.

In Section \ref{sec:Z^d rtimes Z^k} we look at the abelian-by-abelian groups of the form $\Z^d \rtimes \Z^k$, where the action on $\Z^k$ corresponds to multiplication by matrices in an $\R$--split torus inside $\SL{d}$. This of course includes the abelian-by-cyclic groups $\Gamma_M$ where $M$ is a diagonalisable matrix in $\SL{d}$ whose eigenvalues are all real.

We need to introduce a \emph{normal form} for elements of $\Gamma_M$. Firstly, denote by $A_p$ the subgroup of $\Gamma_M$ generated by $\{t^{-p}a_1t^p,\ldots ,t^{-p}a_dt^p\}$ for each integer $p$. Note that the relation $ta_it^{-1}=\varphi_M(a_i)$ implies that for each integer $p$ we have $A_p = \varphi_M(A_{p+1}) \leq A_{p+1}$. Let $\modulus{.}_{A_0}$ denote the word metric on $A_0$ with respect to the generating set $\{ a_1, \ldots , a_d \}$.

We can write each $u \in \Gamma_M$ in the form $t^{-p}u_at^q$ for some $p,q \geq 0$ and $u_a \in A_0$. Furthermore, if $p$ and $q$ are both non-zero, and $u_a$ is in $A_{-1}=\varphi_M(A_0)$, then $u_a=tu_a't^{-1}$ for some $u_a' \in A_0$ and in which case we can re-write $u$ as $u=t^{-(p-1)}u_a't^{q-1}$. Hence we have the following normal form for all elements $u \in \Gamma_M$:
\begin{equation}\label{eq:ab-by-cyc normal form}
u=t^{-p}u_at^q, \textrm{ where $p,q \geq 0$ and if $p,q>0$ then $u_a \notin A_{-1}$.}
\end{equation}

Note that by rewriting elements of $A_0$ in additive notation, one can see how repeated applications of the automorphism $\varphi_M$ correspond to taking a power of $M$. That is, for all $k \in \Z$,
\begin{equation}\label{eq:ab-by-cyc:action of varphi_M on A_0}
\varphi_M^k=\varphi_{M^k}.
\end{equation}

The subgroup $A$ of $\Gamma_M$, given in its defining short exact sequence, is normally generated by $\{a_1, \ldots , a_d\}$. In general $A$ may not be finitely generated, and indeed it will never be when $\det M>1$. As an example, when $d=1$ and $M = (m)$ we get $\Gamma_M=\BS{m}=\Z[\frac{1}{m}]\rtimes_m\Z$, a solvable Baumslag--Solitar group. In this case $A=\Z[\frac{1}{m}]$, which is generated by $\{a=a_1, t^{-1}at,t^{-2}at^2, \ldots \}$ but not by any finite set.

\subsection{The Farb-Mosher space $X_M$}

Suppose that $M\in \GL{d}$ is diagonalisable with all eigenvalues greater than $1$. In \cite{FM01} Farb and Mosher describe a geodesic metric space $X_M$ which is quasi-isometric to $\Gamma_M$. In short, the space $X_M$ can be recognised as the horocyclic product of a $(\det M + 1)$--valent tree $T_M$ and the Lie group $G_M=\R^d \rtimes_M \R$. For a discussion on horocyclic products, in particular horocyclic products of graphs, we refer the reader to \cite{BNW08}. The geometry of $G_M$ has been studied in \cite{DP11}. In order to define $X_M$ we require that $M$ lies on a one-parameter subgroup of $\GL{d}$, hence the requirement that all eigenvalues of $M$ are greater than $1$.

Note that if $M$ does not lie on a one-parameter subgroup but is diagonalisable and has eigenvalues whose absolute values are all greater than $1$, then we may instead consider $\Gamma_{M^2}$, which is an index $2$ subgroup of $\Gamma_M$. We can then define $X_{M^2}$, to which $\Gamma_M$ will be quasi-isometric. Until Lemma \ref{lem:ab-by-cyc distortion}, however, we will assume that all eigenvalues of $M$ are positive.

On $T_M$ we take the path metric $d_T$ with edges assigned length $1$. Fix a basepoint $o \in T_M$ and consider any geodesic ray $\rho_T$ emerging from $o$. This ray determines an ideal point $\omega_T$ in the boundary of $T_M$. For any pair of vertices $x$ and $y$ in $T_M$, the two rays, both asymptotic to $\rho_T$, emerging from $x$ and $y$ respectively will merge at some vertex. We call this vertex the \emph{greatest common ancestor} of $x$ and $y$ and denote it $x \curlywedge y$. Define a height function $\mathfrak{h}:T_M \to \R$ by 
		\[\mathfrak{h}(x)=d_{T}(x,o\curlywedge x)-d_{T}(o\curlywedge x,o).\]
This height function plays a crucial role in the definition of the horocyclic product determining $X_M$. A Busemann function on $G_M$ is the last ingredient necessary to define $X_M$.

\begin{figure}[t!]
\labellist \hair 5pt \small
	\pinlabel $\omega_T$ [b] at 220 413
	\pinlabel $o$ [r] at 48 89
	\pinlabel $x$ [r] at 202 377
	\pinlabel $\mathfrak{h}(x)=-2$ [r] at 0 377
	\pinlabel $\mathfrak{h}(o)=0$ [r] at 0 89
\endlabellist
\centering\includegraphics[width=8cm]{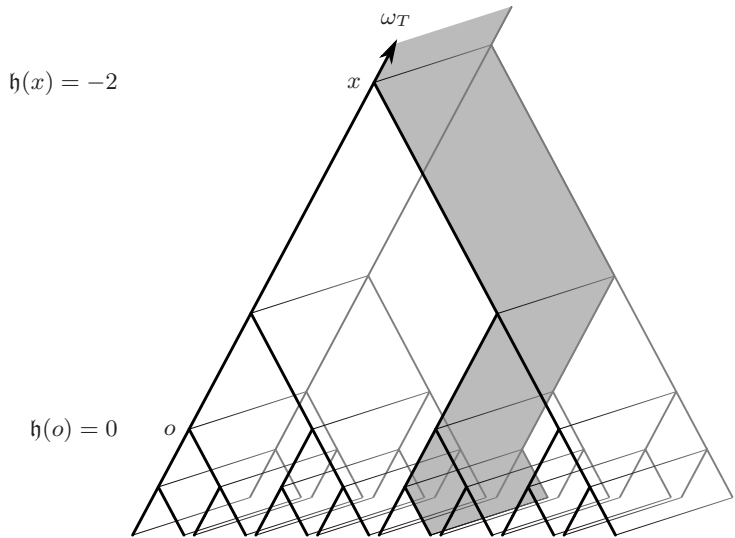}
\caption{Treebolic space, $X_M$ when $M=(2)$. The shaded region represents an example of a coherent hyperplane.}\label{fig:treebolic}
\end{figure}

Given an element $g \in G_M$, we will write it as $g=(\mathbf{g},t_g)$, where $\mathbf{g} \in \R^d$ and $t_g \in \R$. The Busemann function of $G_M$ that we use is that associated to the boundary point $\omega_G$ determined by the ray $\rho_G(t)=(\mathbf{0},t)$, for $t \geq 0$. For a point $(\mathbf{g},t_g) \in G_M$ it takes the value $-t_g \in \R$.

The space $X_M$ is the \emph{horocyclic product} of $T_M$ and $G_M$, with respect to the height function $\mathfrak{h}$ defined on $T_M$ and the Busemann function just defined on $G_M$. It is the subspace of the product space $T_M \times G_M$ consisting of those $(x,g) \in T_M \times G_M$ such that if $g=(\mathbf{g},t_g)$ then $\mathfrak{h}(x)+t_g=0$. We take the metric $d_X$ to be the length metric induced on $X_M$ by the $\ell^\infty$--metric on $T_M \times G_M$.

To give some intuition, very roughly speaking, both the height function and Busemann function tell you how far a point is from the respective ideal points, $\omega_T$ and $\omega_G$. If a point in $T_M \times G_M$ lies in the horocyclic product, and its $T_M$ coordinate lies ``a long way'' from $\omega_T$, then its $G_M$ coordinate must lie ``close'' to $\omega_G$.

Consider the two projections $\pi_G:X_M \to G_M$ and $\pi_T:X_M\to T_M$. Given a bi-finite line $\ell$ in $T_M$, with one end of $\ell$ asymptotic to $\omega_T$, the pre-image $\pi_T^{-1}(\ell)$ is isometric to $G_M$. Farb and Mosher \cite{FM01} call these \emph{coherent hyperplanes} in $X_M$.

\begin{example}
When $d=1$ and $M$ is the $1 \times 1$ matrix $(m)$ we get the solvable Baumslag--Solitar group $\BS{m}$. The Farb-Mosher space associated to $\BS{m}$ is the so-called ``treebolic'' space --- it is the horocyclic product of $T_M$, an $m+1$--valent tree, and $G_M$, which will be the hyperbolic plane with a rescaled metric (rescaled so that the horocycles line up suitably). Figure \ref{fig:treebolic} shows a portion of this space, using the upper-half plane model for the hyperbolic plane. The ideal point $\omega_G$ will be $\infty$. In Figure 1, the reader may picture $\omega_T$ in the direction of top of the page and $\omega_G$ towards the bottom. Thus, in the coherent hyperplanes, the upper-half plane model is ``up-side down'' to how it is usually drawn.
\end{example}

\begin{lemma}\label{lem:horocyclic metric bounds}
Let $M$ be a diagonalisable matrix with all eigenvalues greater than $1$. For any pair of points $(x,g),(y,h) \in X_M$,
		\[\max\{d_G(g,h),d_T(x,y)\} \leq d_X\big( (x,g) , (y,h) \big) \leq d_G(g,h)+d_T(x,y).\]
\end{lemma}

\begin{proof}
Let $\sigma$, denote a geodesic in $X_M$ from $(x,g)$ to $(y,h)$. Let $\sigma_T=\pi_T(\sigma)$ be the projection of the geodesic into the tree component $T_M$, and similarly $\sigma_G=\pi_G(\sigma)$. Since the projections do not increase lengths of paths, the lower bound in the Lemma is straight-forward to obtain by looking at the lengths of $\sigma_T$ and $\sigma_G$.

We now focus on obtaining the upper bound. Since $\sigma_T$ is a path from $x$ to $y$ it must at some point pass through the common ancestor $x \curlywedge y$. Let $c=\mathfrak{h}(x\curlywedge y)$. Notice that
		\[d_T(x,y) = \mathfrak{h}(x)+\mathfrak{h}(y)-2c.\]
Write $g=(\mathbf{g},t_g)$ and $h=(\mathbf{h},t_h)$, where $\mathbf{g,h}\in\R^d$ and $t_g,t_h \in \R$. Then we can draw a piecewise geodesic path from $(x,g)$ to $(y,h)$ by travelling via $(x\curlywedge y, g)$ and $(x \curlywedge y, h)$. By setting $g'=(\mathbf{g},-c)$ and $h'=(\mathbf{h},-c)$, the triangle inequality then gives
\begin{eqnarray*}
d_X\big( (x,g) , (y,h) \big) & \leq & \mathfrak{h}(x)-c + d_G(g',h') + \mathfrak{h}(y)-c\\
		& = & d_T(x,y)+d_G(g',h').
\end{eqnarray*}
The last thing to observe here is that $t_g=-\mathfrak{h}(x) \leq -c$, and similarly $t_h \leq -c$. Hence $d_G(g,h) \geq d_G(g',h')$, since $M$ is diagonalisable over $\R$ and ${\lambda}>1$ for each eigenvalue $\lambda$ of $M$. This proves the upper bound.
\end{proof}

To make our calculations in Lemma \ref{lem:ab-by-cyc distortion} more precise we will use a metric $d_L$ on $G_M$, as described in \cite{DP11}, which is bi-Lipschitz equivalent to $d_G$. To motivate the metric $d_L$, we consider the case when $M=(m)$ and $G_M$ is the rescaled hyperbolic plane (we picture it with the upper-half plane model, and $\infty$ vertically upwards). Consider a geodesic in $G_M$ connecting points $x$ and $y$. It will either be a vertical line, when $x$ and $y$ lie one above the other, or a circular arc with centre on the horizontal axis (although the rescaling of the metric means it will not, strictly speaking, be circular). We can approximate the latter type of geodesic by using a path which consists of 2 vertical geodesics and one horizontal path. The two vertical segments will start at $x$ and $y$ respectively and take us upwards to a common horocycle. This horocycle will be the lowest horocycle where the two vertical geodesics are separated by at most distance 1 (if it is strictly less than 1, then one vertical geodesic will have length 0). The horizontal segment will then connect the two top ends of these geodesics creating a path from $x$ to $y$. The length of this path gives us an estimate for $d_G(x,y)$.

\begin{defn}For each horocycle $\R^n \times\{t\}$ in $G_M$ consider the metric
		\[d_{t,M}\big((a,t),(b,t)\big) := \norm{M^{-t}(b-a)}\]
where $\norm{.}$ is  the $\ell^1$--norm on $\R^d$. Given $(a,t_a),(b,t_b) \in \Gamma_M$, let $t_0\in \R$ be minimal such that $d_{t_0,M}\big((a,t_0),(b,t_0)\big)\leq 1$. If $t_0\geq t_a,t_b$ then define
		\[d_L\big( (a,t_a),(b,t_b)\big) := \modulus{t_a-t_0}+\modulus{t_0-t_b} + 1.\]
On the other hand, if $t_a\geq t_b,t_0$ then set
		\[d_L\big( (a,t_a),(b,t_b)\big) := t_a-t_b + \norm{M^{-t_a}(b-a)}.\]
\end{defn}

The metric $d_L$ will be bi-Lipschitz equivalent to $d_G$. In particular, let $K_M>0$ be such that
		$$\frac{1}{K_M}d_G(g,h) \leq d_L(g,h) \leq K_Md_G(g,h).$$

We can recognise $\Gamma_M$ as the HNN extension of $A_0=\langle a_1, \ldots a_d\rangle$ by the stable letter $t$. The corresponding Bass-Serre tree will then be $T_M$ and we will identify the basepoint vertex $o$ in $T_M$ with the vertex of the Bass-Serre tree which is stabilised by $A_0$. The action of $\Gamma_M$ on $T_M$ is then given by the usual action of an HNN extensions on its Bass-Serre tree.

Note that we may embed $\Gamma_M$ in $G_M$ via the homomorphism which sends the generators $a_1,\ldots ,a_d$ of $A_0$ to the standard integer basis elements of $\R^d\times \{0\}$ and sends $t$ to $(0,1)$. Denote the image of $u \in \Gamma_M$ under this homomorphism by $\overline{u}$.

\begin{lemma}\label{lem:ab-by-cyc distortion}
Suppose $M$ is diagonalisable over $\R$ with all eigenvalues of absolute value greater than $1$. Let $u \in A$ and suppose $p$ is the minimal non-negative integer such that $u \in A_p$. Let $u_a \in A_0$ be such that $u=t^{-p}u_at^p$. Then
\begin{enumerate}[label=(\roman{enumi})]
\item\label{item:ab-by-cyc metric lower bound}		$ p \preceq \modulus{u}$;
\item\label{item:ab-by-cyc metric bounds}		$\log\modulus{u_a}_{A_0} \preceq \modulus{u} \preceq  \max\{\log \modulus{u_a}_{A_0} , p \}$;
%\item\label{item:ab-by-cyc distortion functions}		$\delta_{A_0}^{\Gamma_M}(n) \asymp \exp(n) \asymp \ldist_{A_0}^{\Gamma_M}(n)$,
\end{enumerate}
where the constants involved in each $\preceq$ or $\asymp$ relation depend only on the choice of $M$.
\end{lemma}

\begin{remark}
A consequence of Lemma \ref{lem:ab-by-cyc distortion} (ii) is that the distortion functions, $\delta_{A_0}^{\Gamma_M}$ and $\mathfrak{d}_{A_0}^{\Gamma_M}$, with respect to the metric $\modulus{.}_{A_0}$, are both exponential. However, one should note that we will calculate the twisted conjugacy length function of $A$ using the restriction of the word metric on $\Gamma_M$. Thus, for the purposes of applying Corollary \ref{cor:CLF for group extension with quotient Z}, both distortion functions $\delta_A^{\Gamma_M}$ and $\mathfrak{d}_A^{\Gamma_M}$ are linear.
\end{remark}

\begin{proof}
If $M$ has negative eigenvalues then we need to consider instead $X_{M^2}$. Since $\Gamma_{M^2}$ is the index $2$ subgroup of $\Gamma_M$ generated by $\{a_1,\ldots,a_d,t^2\}$ we know in particular that $A$ is contained in $\Gamma_{M^2}$. It is enough therefore to obtain an estimate for the length of elements of $A$ with respect to the generating set for $\Gamma_{M^2}$. Hence in the following we assume that $M$ has no negative eigenvalues.

The element $u$ sends the basepoint $x=(o,e) \in X_M$ to $ux=(uo,\overline{u})$. The action of $u$ on $T_M$ means that $d_T(o,uo) =2p$. This follows from the fact that $u \in A_k$ if and only if $ut^{-k}o = t^{-k}o$. Hence $o \curlywedge uo = t^{-p}o$, and any path from $o$ to $uo$ must pass through the vertex $t^{-p}o$. By Lemma \ref{lem:horocyclic metric bounds}, this implies \ref{item:ab-by-cyc metric lower bound}.

\begin{figure}[h!]
\labellist \hair 5pt \small
	\pinlabel $o$ [r] at 3 21
	\pinlabel $t^{-(p-1)}o$ [r] at 115 245
	\pinlabel $ut^{-p}o=t^{-p}o$ [r] at 163 341
	\pinlabel $u$ [t] at 160 232
	\pinlabel $uo$ [l] at 324 21
\endlabellist	
\centering\includegraphics[width=4cm]{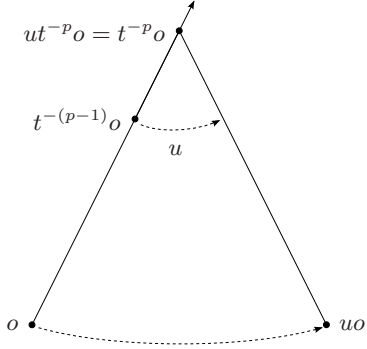}
\caption{The path from $o$ to $t^{-p}o$ will be mapped under $u$ to a path from $uo$ to $t^{-p}o$. The concatenation of these paths gives us a path from $o$ to $uo$. The choice of $p$ tells us that it is a geodesic, since $ut^{-(p-1)}o\neq t^{-(p-1)}o$.}\label{fig:tree_bass-Serre_path}
\end{figure}

Suppose $u_a=a_1^{u_1}\ldots a_d^{u_d}$. Then $\overline{u}=(M^{-p}\mathbf{u},0)$, where $\mathbf{u}=(u_1,\ldots,u_d)$. Let $t_0 = \inf \{ t \geq 0 \mid \norm{M^{-t-p}\mathbf{u}}\leq 1\}$. Firstly, supposing that $t_0 \neq 0$ leads us to 
		\[d_L(e,\overline{u})=2t_0+2p+1.\]
Since the eigenvalues of $M$ lie outside the unit circle, $M$ is expanding, hence $\norm{M^{-t-p}\mathbf{u}}$ must lie in the interval $[\lambda_2^{-t-p}\norm{\mathbf{u}},\lambda_1^{-t-p}\norm{\mathbf{u}}]$, where $\lambda_2$ is the smallest among absolute values of eigenvalues of $M$ and $\lambda_1$ the largest. Hence in particular
		\[\lambda_2^{-t_0-p}\norm{\mathbf{u}} \leq 1 \leq \lambda_1^{-t_0-p}\norm{\mathbf{u}}\]
implying that
		\[\log_{\lambda_2}\norm{\mathbf{u}} \leq t_0+p \leq \log_{\lambda_1}\norm{\mathbf{u}}.\]
So using the above values for $d_L(e,\overline{u})$ we get
		\[\frac{1}{K_M}\left(2\log_{\lambda_2}\norm{\mathbf{u}} +1\right) \leq \frac{1}{K_M}d_L(e,\overline{u})\leq d_G(x,ux)\]
		\[d_G(x,ux)  \leq K_Md_L(e,\overline{u}) \leq K_M(2\log_{\lambda_1}\norm{\mathbf{u}}+1)\]
where $K_M$ is the bi-Lipschitz constant between $d_L$ and $d_G$. An application of Lemma \ref{lem:horocyclic metric bounds} and insertion of the quasi-isometry constants between $X_M$ and $\Gamma_M$ will deal with \ref{item:ab-by-cyc metric bounds} in the case when $\norm{M^{-p}\mathbf{u}}>1$.

Now suppose that $t_0=0$. This implies that $\norm{M^{-p}\mathbf{u}} \leq 1$ and $d_L(e,\overline{u})\leq 1$. From the former we get $\lambda_2^{-1} \norm{\mathbf{u}} \leq 1$ and hence $\log_{\lambda_2}\norm{\mathbf{u}}\leq p$. The latter, alongside Lemma \ref{lem:horocyclic metric bounds}, implies
		\[ \max\{2p, 1\} \leq d_X(x,ux) \leq 2p+1.\]
Thus part \ref{item:ab-by-cyc metric bounds} holds. 
%Finally note that part \ref{item:ab-by-cyc distortion functions} follows immediately from \ref{item:ab-by-cyc metric bounds}.
\end{proof}

\begin{figure}[h!]
\labellist \hair 5pt \small
	\pinlabel $x$ [t] at 31 22
	\pinlabel $ux$ [t] at 341 156
	\pinlabel $p$ [r] at 40 100
	\pinlabel $t_0$ [r] at 103 257
\endlabellist	
\centering\includegraphics[width=7cm,height=7cm]{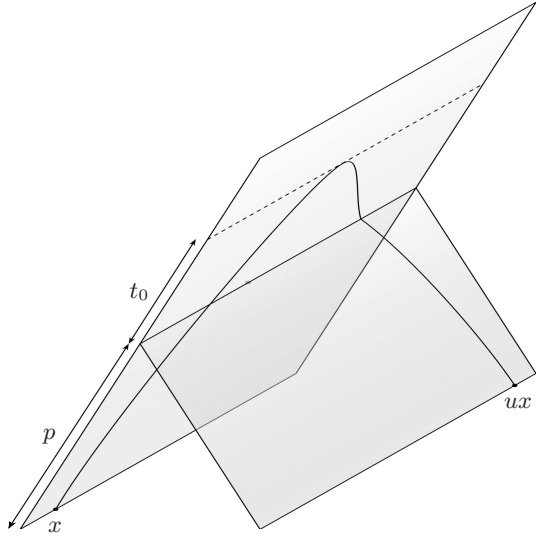}
\caption{In the case of $\BS{m}$, when $X_M$ is treebolic space, a geodesic from $x$ to $ux$ will have to reach a height of $p$ above $x$ in order to switch into the appropriate coherent hyperplane. We can project the path into one copy of the hyperbolic plane (more generally we project onto a copy of $G_M$). If the projected path is a geodesic then the value $t_0$ in the proof of Lemma \ref{lem:ab-by-cyc distortion} will be non-negative and $p+t_0$ gives a measure of the maximum height reached by the geodesic in the hyperbolic plane.}\label{fig:path_in_X_M}
\end{figure}

\subsection{Restricted conjugacy length function of $A$ in $\Gamma_M$}

We will first find a control on the restricted conjugacy length function of $A$ in $\Gamma_M$ when $M$ is diagonalisable and acts on each eigenspace by expansion.

\begin{prop}\label{prop:ab-by-cyc RCLF}
Suppose $M$ is diagonalisable with all eigenvalues having absolute value greater than $1$. Then the restricted conjugacy length function of $A$ in $\Gamma_M$ satisfies
		\[\RCLF_A^{\Gamma_M}(n) \preceq n.\]
\end{prop}

\begin{proof}
Suppose $u$ and $v$ are distinct elements in $A$ which are conjugate in $\Gamma_M$. Let $p,q$ be minimal non-negative integers such that $u \in A_p$ and $v \in A_q$. Since $A$ is abelian, $u,v$ must be conjugate by $t^k$ for some integer $k$. By reversing the roles of $u,v$ if necessary, we may assume that $k$ is non-negative and that $u=t^kvt^{-k}=\varphi^k_M(v)$. Since $\varphi_M(A_q) \subseteq A_{q-1}$, we see that $u \in A_{q-k}$. But by minimality of our choice of $p$ we have either 
\begin{enumerate}
\item\label{item:ab-by-cyc:RCLF 1} $p \leq q-k$, so $k \leq q-p$; or
\item\label{item:ab-by-cyc:RCLF 2} $q-k < 0 = p$.
\end{enumerate}
Case (\ref{item:ab-by-cyc:RCLF 1}) can be dealt with using part \ref{item:ab-by-cyc metric lower bound} of Lemma \ref{lem:ab-by-cyc distortion}. Thus from the first situation we get $\modulus{t^k} \preceq \modulus{v}$ and so we have a linear control on the conjugator length between $u$ and $v$.

Case (\ref{item:ab-by-cyc:RCLF 2}) can only occur if $p=0$ and $k > q$. Suppose $v$ can be written as $t^{-q}v_at^q$ and $v_a=a_1^{v_1}\ldots a_d^{v_d}$. Then
		\[u=t^{k-q}v_at^{q-k}=\varphi_M^{k-q}(v_a).\]
If $u=a_1^{u_1}\ldots a_d^{u_d}$, then equation (\ref{eq:ab-by-cyc:action of varphi_M on A_0}) tells us that
		\[(u_1,\ldots ,u_d)\transpose=M^{k-q} \cdot (v_1 , \ldots , v_d)\transpose.\]
This gives us that $\modulus{u}_{A_0} \geq \lambda_2^{k-q}\modulus{v_a}_{A_0}\geq \lambda_2^{k-q}$, where $\lambda_2$ is equal to the minimal absolute value of an eigenvalue of $M$. Then, by part \ref{item:ab-by-cyc metric bounds} of Lemma \ref{lem:ab-by-cyc distortion}, we get
		\[k \preceq \modulus{u}+q.\]
Hence, by a final application of part \ref{item:ab-by-cyc metric lower bound} of Lemma \ref{lem:ab-by-cyc distortion} we obtain the linear upper bound for the restricted conjugacy length function of $A$ in $\Gamma_M$.
\end{proof}

\subsection{Twisted conjugacy in $A$}

Maintaining the assumption that $M$ is diagonalisable and has eigenvalues with absolute value greater than $1$ we obtain the following result regarding twisted conjugacy in the subgroup $A$.

We remind the reader that the action of $\varphi_M$ on $A$ is via conjugation by $t$ and, since $A$ is closed under conjugation by $t$ and $t^{-1}$, $\varphi_M^i$ is an automorphism of $A$ for each $i \in \Z$.

\begin{prop}\label{prop:ab-by-cyc TCLF}
Suppose $M$ is diagonalisable with all eigenvalues having absolute value greater than $1$. Then there exists $L_M>0$ such that for all $i \in \Z$ the twisted conjugacy length function of $A$, with respect to the metric on $A$ given by restricting the word metric on $\Gamma_M$ to $A$, satisfies
		\[\TCLF_A(n;\varphi_M^i) \preceq n+\modulus{i}.\]
\end{prop}

\begin{proof}
Let $u \in A_p$, $v \in A_q$ and $x \in A_y$ with $p,q,y$ minimal such non-negative integers. Suppose they satisfy the twisted conjugacy relationship $u \varphi_M^i (x) = xv$, for some $i \neq 0$ (note that the case when $i=0$ reduces to the conjugacy length function of an abelian group).  This is equivalent to
\begin{equation}\label{eq:ab-by-cyc twisted relation}
t^{-p}u_at^pt^{i-y}x_at^{y-i}=t^{-y}x_at^yt^{-q}v_at^q
\end{equation}
where $u=t^{-p}u_at^p,v=t^{-q}v_at^q$ and $x=t^{-y}x_at^y$. Let $u_a = a^{u_1}\ldots a_d^{u_d}$, $v_a = a^{v_1}\ldots a_d^{v_d}$ and $x_a = a^{x_1}\ldots a_d^{x_d}$. We will let $\mathbf{u}=(u_1,\ldots ,u_d)\transpose$, $\mathbf{v}=(v_1,\ldots ,v_d)\transpose$ and $\mathbf{x}=(x_1,\ldots ,x_d)\transpose$. We can rewrite equation \eqref{eq:ab-by-cyc twisted relation} using vectors:
		\[M^{-p}\mathbf{u}+M^{i-y}\mathbf{x} = M^{-y}\mathbf{x}+M^{-q}\mathbf{v}.\]
Using the fact that $1$ is not an eigenvalue of $M$ we can rearrange this to give
		\[\mathbf{x}=\big(M^i-1\big)^{-1}\big(M^{y-q}\mathbf{v}-M^{y-p}\mathbf{u}\big).\]
Making use of the fact that all eigenvalues of $M$ have absolute value greater than $1$, we therefore get that $\modulus{x_a}_{A_0} \preceq \exp(y+\modulus{i})$, which, by Lemma \ref{lem:ab-by-cyc distortion}, implies that $\modulus{x}\preceq y+\modulus{i}$. Hence we need to obtain an upper bound on $y$.

We claim that $y \leq n+i$. If this was not true then in particular $y > q$ and $y>p+i$. We can rearrange equation \eqref{eq:ab-by-cyc twisted relation} by shuffling the terms in $t$ that appear in the middle of each side to one end:
		\[t^{i-y}\varphi_M^{y-p-i}(u_a)x_at^{y-i}=t^{-y}x_a\varphi_M^{y-q}(v_a)t^y.\]
Note that both sides of this equation are in normal form, as described in \eqref{eq:ab-by-cyc normal form}. Indeed, we can see this as follows: since $y-p-i>0$ and $y-q>0$, both $\varphi_M^{y-p-i}(u_a)$ and $\varphi_M^{y-q}(v_a)$ lie in the image $\varphi_M(A_0)$. However, $y>0$ implies that $x_a \notin \varphi_M(A_0)$, implying that neither $\varphi_M^{y-p-i}(u_a)x_a$ nor $x_a\varphi_M^{y-q}(v_a)$ can lie in $\varphi_M(A_0)$. Hence both sides of this equation are in normal form. But this is only possible if $y=y-i$ and, in particular, we would have to contradict our assumption that $i\neq 0$. We can therefore deduce that $y \leq n+i$ and by the argument above we have that $\modulus{x}\preceq n+ \modulus{i}$.
\end{proof}

\subsection{Conjugacy length in $\Gamma_M$}

We now obtain the upper bound for the conjugacy length function of $\Gamma_M$ when $M$ is a diagonalisable matrix with all eigenvalues greater than $1$.

\begin{proof}[Proof of Theorem~{\ref{thm:clf of ab-by-cyc M expanding}}]
This is a straight-forward application of Corollary \ref{cor:CLF for group extension with quotient Z}. The distortion function $\delta_A^{\Gamma_M}$ is the identity map since we calculated the twisted conjugacy length function in $A$ with respect to the word metric on $\Gamma_M$. Similarly $\ldist_A^{\Gamma_M}$ is also linear, so $\invdist_A^{\Gamma_M}$ will be as well. The set $A_G^{(n)}$ of automorphisms of $A$ that we are to consider consists of those automorphisms $\varphi_x$, where $x \in \Gamma_M$ with $\modulus{x}\leq n$ and $\varphi_x(a)=xax^{-1}$ for $a \in A$. This clearly includes all the automorphisms $\varphi_M^i$ for $\modulus{i} \leq n$. What's more, if $x=t^k a$ for some $a \in A$ and $k \in \Z$, then $\varphi_x=\varphi_{t^k}$. Hence
		\[A_G^{(n)}=\{\varphi_M^i \mid \modulus{i}\leq n\}.\]
Theorem \ref{thm:clf of ab-by-cyc M expanding} now follows by feeding Propositions \ref{prop:ab-by-cyc RCLF} and \ref{prop:ab-by-cyc TCLF} into Corollary \ref{cor:CLF for group extension with quotient Z}.
\end{proof}

\section{Semidirect products $\Z^n \rtimes \Z^k$}\label{sec:Z^d rtimes Z^k}

We will now turn our attention to a class of polycyclic abelian-by-abelian groups \mbox{$\Z^d \rtimes \Z^k$,} where $d>k$. We will express an element of $\Gamma=\Z^d \rtimes \Z^k$ as a pair $(x,y)$, with $x \in \Z^d$ and $y \in \Z^k$. The action of $\Z^k$ on $\Z^d$ is via matrices in an $\R$--split torus in $\SL{d}$. That is, the semidirect product is defined by $\varphi : \Z^k \to \SL{d}$ such that the image $\varphi(\Z^k)$ consists of matrices which are simultaneously diagonalisable over $\R$, all of whose eigenvalues are positive. Thus we can choose a basis of $\R^d$ which consists of common eigenvectors for the matrices in the image $\varphi(\Z^k)$.

Throughout we will let $\norm{.}$ denote the $\ell^1$ norm on either $\Z^d$ or $\Z^k$. We begin by giving a method to relate the size of a component of $u \in \Z^d$, with respect to one of the eigenvectors, to the size of $\norm{u}$.

\begin{lemma}\label{lem:log component is bounded}
Let $u=(u_1,\ldots ,u_d) \in \R^d$, with the coordinates given with respect to a basis of eigenvectors of the matrices in $\varphi(\Z^k)$. Then there exists a constant $\alpha_\varphi$ such that for each $i$ such that $u_i \neq 0$
		\[\modulus{\log (\modulus{u_i})} \leq \alpha_\varphi \log \norm{u}.\]
\end{lemma}

\begin{proof}
We will prove the Lemma for $i=1$. Let $E_1, \ldots , E_d$ be the one-dimensional spaces spanned by each eigenvector for the matrices in $\varphi(\Z^k)$. Then $\modulus{u_1}$ corresponds to the distance from $u$ to the hyperplane $E_2 \oplus \ldots \oplus E_d$ in a direction parallel to $E_1$. In order to obtain a lower bound on $\log \modulus{u_1}$ we need to find a lower bound on the distance from $u$ to $E_2 \oplus \ldots \oplus E_d$. This lower bound follows from the subspace theorem of Schmidt \cite{Schm72}. Since $E_2, \ldots , E_m$ are eigenspaces for $\varphi(y) \in \SL{d}$, there exists algebraic numbers $\alpha_1,\ldots,\alpha_d$ such that $E_2 \oplus \ldots \oplus E_d= \{ x \in \R^d \mid x \cdot (\alpha_1, \ldots , \alpha_d)=0 \}$. Then, for $u \in \Z^d$, we have
		\[d(u,E_2\oplus \ldots \oplus E_d)=\frac{\modulus{u\cdot(\alpha_1,\dots,\alpha_d)}}{\norm{(\alpha_1,\ldots,\alpha_d)}}\textrm{.}\]
Thus, by the subspace theorem, for every $\varepsilon > 0$ there exists a positive constant $C$ such that for every $u \in \Z^d$ we have the following bound on the distance to the hyperplane:
		\[d(u,E_2\oplus \ldots \oplus E_d) \geq \frac{C}{\norm{u}^{d-1+\varepsilon}}\textrm{.}\]
In particular this gives us a lower bound on $\modulus{u_1}$ and hence 
		\[-\log \modulus{u_1} \leq (d-1+\varepsilon)\log \norm{u} - \log C\textrm{.}\]

By a trigonometric argument we can obtain an upper bound on $\log \modulus{u_1}$ which will depend on the angles between the eigenspaces. Hence, combining this with the lower bound, there exists a positive constant $\alpha_\varphi$, determined by $d$, $\varepsilon$ and $\varphi$, such that
		$\modulus{\log(\modulus{u_1})} \leq \alpha_\varphi \log \norm{u}.$
\end{proof}

It is important to understand the distortion of the $\Z^d$ component in $\Gamma$. The following two lemmas give us a handle on this. The second, Lemma \ref{lem:upper bound - distortion in Z^d rtimes Z^r}, gives an exponential upper bound for the distortion function while the first, Lemma \ref{lem:lower bound - distortion in Z^d rtimes Z^k}, shows that for any element of the form $(x,0)$ in $\Gamma$ we can take a significant shortcut to get to $x$ from $0$ by using the action of one of the matrices in $\varphi(\Z^k)$.

\begin{lemma}\label{lem:lower bound - distortion in Z^d rtimes Z^k}
Suppose that $\varphi(\Z^k)$ contains a matrix whose eigenvalues are all distinct from $1$. Then there exists a constant $A_\varphi >0$ such that for every $x \in \Z^d$
		\[\modulus{(x,0)}_\Gamma \leq A_\varphi \log \norm{x}.\]
In particular the inverse subgroup distortion function is bounded above by a logarithm:
		\[\invdist_{\Z^d}^\Gamma(n)\leq A_\varphi \log(n).\]
\end{lemma}

\begin{proof}
The bound on $\overline{\mathfrak{d}}_{\Z^d}^\Gamma$ is an application of Lemma \ref{lem:invdist is minimal}, using the bound on $\modulus{(x,0)}_\Gamma$ found in the first part of the Lemma. To obtain this bound, we will use that fact that $\Gamma = \Z^d \rtimes_\varphi \Z^k$ is a uniform lattice in $G=\R^d \rtimes_\varphi \R^k$, finding an upper bound for $\modulus{(x,0)}_G$ and hence also for $\modulus{(x,0)}_\Gamma$. Here, by $\modulus{(x,0)}_G$ we mean the distance $d_G(1,(x,0))$ where $d_G$ is a left-invariant Riemannian metric on $G$.

First, let $y$ be a minimal length vector in $\Z^k$ with the property that $\varphi(y)$ has no eigenvalues equal to $1$.  Write $x$ in coordinates $(x_1,\ldots ,x_d)$ with respect to a basis of eigenvectors for $\varphi(y)$, which have been chosen so that $x_i \geq 0$ for each $i$. Let $\tilde{e}_i$ denote the $i$--th eigenvector, in this notation it will have $0$'s everywhere except in the $i$--th coordinate where we put a $1$. Let $\lambda_i$ be the eigenvalue of $\varphi(y)$ corresponding to $\tilde{e}_i$. For $i=1,\ldots,d$ let $\alpha_i \in \R$ be such that $\lambda_i^{\alpha_i}=x_i$. Then
		\[(x,0)=(0,\alpha_1y)(\tilde{e}_1,\alpha_2y-\alpha_1y)(\tilde{e}_2,\alpha_3y-\alpha_2y)\ldots(\tilde{e}_d,-\alpha_dy).\]
Calculating the length of each term in the product gives us an upper bound on $\modulus{(x,0)}_G$ as
\begin{eqnarray*}
\modulus{(x,0)}_G &\leq& \modulus{\alpha_1}\norm{y}+1+\modulus{\alpha_2-\alpha_1}\norm{y}+1+\modulus{\alpha_3-\alpha_2}\norm{y}+\ldots+1+\modulus{\alpha_d}\norm{y}\\
					&\leq& d + 2\norm{y}\sum_{i=1}^d\modulus{\alpha_i}.
\end{eqnarray*}
But $\alpha_i=\frac{\log(x_i)}{\log(\lambda_i)}$, so by applying Lemma \ref{lem:log component is bounded} we get $\modulus{\alpha_i}\leq \alpha_\varphi \frac{\log \norm{x}}{\modulus{\log(\lambda_i)}}$. Hence
		\[\modulus{(x,0)}_G \leq d + 2\norm{y}\alpha_\varphi\log \norm{x}\sum_{i=1}^d\frac{1}{\modulus{\log (\lambda_i)}}.\]
Combining this with the fact that $\Gamma$ is a uniform lattice in $G$ gives us the existence of the constant $A_\varphi$ given in the Lemma.
\end{proof}

\begin{lemma}\label{lem:upper bound - distortion in Z^d rtimes Z^r}
Suppose the image $\varphi(\Z^k)$ is generated by matrices $\varphi_1 , \ldots , \varphi_k$ and that $\lambda$ is the largest eigenvalue of any of the matrices $\varphi_1,\varphi_1^{-1},\ldots ,\varphi_k,\varphi_k^{-1}$.  Let $x \in \Z^d$. Then
		\[\norm{x} \leq \modulus{(x,0)}_\Gamma\left(\lambda^{\modulus{(x,0)}_\Gamma}+1\right).\]
In particular the subgroup distortion function is bounded above by an exponential:
		\[\delta_{\Z^d}^\Gamma(n) \leq n(\lambda^n+1).\]
\end{lemma}

\begin{proof}
The generators of $\Gamma$ are taken to be the set of elements of the form either $(e_i,0)$ or $(0,e_j)$ where $e_i \in \Z^d$, $e_j \in \Z^k$ are elements of the standard bases. We can write $(x,0)$ as a geodesic word on these generators, grouping together the generators of the form $(e_i,0)$ and the generators of the form $(0,e_j)$:
\begin{eqnarray*}
(x,0)&=&(\alpha_1,0)(0,\beta_1)(\alpha_2,0)\ldots(\alpha_r,0)(0,\beta_r)\\
	 &=&(\alpha_1,\beta_1)(\alpha_2,\beta_2)\ldots(\alpha_r,\beta_r)
\end{eqnarray*}
where $\alpha_i$ and $\beta_i$ are non-zero for all $1\leq i \leq r$, except possibly for $\alpha_1$ and $\beta_r$. First note that $\modulus{(x,0)}_\Gamma=\sum_{i=1}^r(\norm{\alpha_i}+\norm{\beta_i})$. We also obtain the following expression for $x$:
		\[x=\alpha_1 + \varphi(\beta_1)(\alpha_2) + \varphi(\beta_1+\beta_2)(\alpha_3)+\ldots+\varphi(\beta_1+\ldots+\beta_{r-1})(\alpha_r).\]
Since $\norm{\beta_1+\ldots+\beta_i} \leq \modulus{(x,0)}_\Gamma$ for each $i$, we get an upper bound on the norm of $x$:
		\[\norm{x}\leq \norm{\alpha_1} + \lambda^{\modulus{(x,0)}_\Gamma}\left(\norm{\alpha_2}+\ldots+\norm{\alpha_r}\right)\leq \modulus{(x,0)}_\Gamma+\lambda^{\modulus{(x,0)}_\Gamma}\modulus{(x,0)}_\Gamma .\]
Thus the lemma is proved.
\end{proof}

We now give an upper bound on the conjugacy length function of $\Z^d \rtimes \Z^k$. We will see that when $k=1$ we can produce a linear upper bound, but when $k>1$ we have to settle for exponential. The main obstacle that prevents us from finding a better than exponential upper bound is the nature of the projection of centralisers of elements into the $\Z^k$ coordinate. In the language of Theorem \ref{thm:group extensions}, this is the quantity measured by the function $\rho$.

\begin{thm}\label{thm:generalised SOL}
Let $\Gamma = \Z^d \rtimes_\varphi \Z^k$, where the image of $\varphi : \Z^k \hookrightarrow \SLdZ$ is contained in an $\R$--split torus $T$. Then there exist constants $A>1$ and $B>0$ such that
\begin{enumerate}[label={({\arabic*})}]
\item if $k=1$ then $\CLF_\Gamma(n) \leq Bn$;
\item if $k>1$ then $\CLF_\Gamma(n) \leq A^n$.
\end{enumerate}
\end{thm}

\begin{proof}
We will apply Theorem \ref{thm:group extensions}. Since $\Gamma$ is abelian-by-abelian we need to find bounds only for the values of $\RCLF_{\Z^d}^{\Gamma}(n)$, $\TCLF_{\Z^d}(n;\varphi)$ and $\rho(u,v)$, for $(u,v) \in \Gamma$, where $\rho$ is the function as defined in Theorem \ref{thm:group extensions}.

\vspace{2mm}
\noindent \underline{\textsc{Step} 1:} Estimating $\RCLF_{\Z^d}^\Gamma(n)$.
\vspace{2mm}

\noindent Consider $(u,0),(w,0) \in \Z^d$. Then $(u,0)(x,y)=(x,y)(w,0)$ if and only if $u=\varphi(y)(w)$. Note that we can immediately set $x$ to be zero.

Suppose $\varphi(\Z^k)$ is generated by matrices $\varphi_1 , \ldots , \varphi_k$, so that if $y = (y_1,\ldots,y_k)$ then
		\[\varphi(y)=\varphi_{1}^{y_1}\ldots \varphi_{k}^{y_k}\textrm{.}\]
Fix a basis of eigenvectors of the matrices in $T$. With respect to this basis, let $u,v$ be represented with coordinates $(u_1,\ldots,u_d),(w_1,\ldots,w_d)$ respectively. Suppose the eigenvalues of $\varphi_i$ are $\lambda_{j,i}$ for $j = 1, \ldots , d$ and $i=1,\ldots , k$. Then from $u=\varphi(y)(w)$ we get the following system:
		\[u_j= \left( \prod_{i=1}^{k}\lambda_{j,i}^{y_i} \right)w_j\textrm{.}\]
By taking logarithms we see that this system is equivalent to the matrix equation $L{y}={a}$, where $L$ is the $d \times k$ matrix with $(r,s)$--entry equal to $\log \modulus{\lambda_{r,s}}$ and ${a}$ is the vector with $j^{\mathrm{th}}$ entry equal to $\log \modulus{u_j} - \log \modulus{w_j}$. Since the matrices $\varphi_1 , \ldots , \varphi_k$ generate a copy of $\Z^k$, the columns of $L$ are linearly independent. Hence we may take a non-singular $k \times k$ minor $L'$ and we get a matrix equation $L' {y} = {a}'$. By Cramer's Rule, for each $i = 1,\ldots,k$ we have 
		\[y_i = \frac{\det (L^{(i)})}{\det (L')}\]
where $L^{(i)}$ is obtained from $L'$ by replacing the $i^\textrm{th}$ column with ${a}'$. Hence $\modulus{y_i}$ is bounded by a linear expression in the terms $\modulus{\log(\modulus{u_j})}+\modulus{\log(\modulus{w_j})}$, for $j = 1 , \ldots , k$, and the coefficients are determined by the choice of $\varphi$. Therefore, by Lemma \ref{lem:log component is bounded}, we obtain an upper bound for each $\modulus{y_i}$ as linear sum of $\log(\modulus{u})$ and $\log(\modulus{w})$.

To reach an upper bound for the restricted conjugacy length function, we observe that we are able to find $y \in \Z^k$ such that $(u,0)(0,y)=(0,y)(w,0)$ and $\norm{y} \leq B_1 (\log \norm{u} +\log \norm{w})$ for some constant $B_1>0$, determined by $\varphi$ and independent of $u,w$. Furthermore, because the first coordinate in $\Gamma$ is exponentially distorted, in particular by using Lemma \ref{lem:upper bound - distortion in Z^d rtimes Z^r}, we indeed have a linear upper bound on conjugator length:
		\[\modulus{(0,y)}_\Gamma \leq \norm{y} \leq B_2n\]
for some $B_2>0$ independent of $u,w$, and where $n=\modulus{(u,0)}_\Gamma + \modulus{(w,0)}_\Gamma$. Thus the restricted conjugacy length function is at most linear:
		\[\RCLF_{\Z^d}^\Gamma(n)\leq B_2n.\]

\vspace{2mm}
\noindent \underline{\textsc{Step} 2:} Estimating $\TCLF_{\Z^d}(n;\varphi(v))$.
\vspace{2mm}

\noindent This is precisely the situation dealt with in Example \ref{example:TCL for abelian}, where we take $\varphi(v)$ in place of $\varphi$. It gives us
		\[\TCLF_{\Z^d}(n;\varphi(v)) \leq (1 +\lambda_v)n\]
where $\lambda_v$ is the largest eigenvalue of $\varphi(v)$. Suppose $\lambda$ is the largest eigenvalue of any of the generating matrices $\varphi_1,\ldots,\varphi_k$ or their inverses. Then $\lambda_v \leq \lambda^{\norm{v}}$.

\vspace{2mm}
\noindent \underline{\textsc{Step} 3:} Estimating $\rho(u,v)$.
\vspace{2mm}

\noindent If $k=1$ then we are in the case of a cyclic extension, so we need only recall that $\rho(u,v)\leq n$ will suffice here. Now suppose $k>1$. Note that $(a,b)$ is in the centraliser $Z_\Gamma(u,v)$ if and only if 
		\[a=(\Id-\varphi (v))^{-1}(\Id -  \varphi (b))u \in \Z^d\textrm{.}\]
We will show that given any $b \in \Z^k$, there exists a constant $m$ which is bounded by an exponential in $\norm{v}$ and such that $(\Id-\varphi (v))^{-1}(\Id -  \varphi (mb))u \in \Z^d$. This will give an exponential upper bound on $\rho(u,v)$.

Let $L:=(\Id - \varphi (v))\Z^d$. Denote by $c$ the absolute value of the determinant of $\Id-\varphi (v)$. Then $c$ is the index of $L$ in $\Z^d$. Since $\varphi(b)$ commutes with $\varphi(v)$, and hence $\Id-\varphi(v)$, it follows that $\varphi(b)L=L$. Let $\bar{u}$ be the image of $u$ in $\Z^d/L$. Then there exists some $m \leq c$ such that $\varphi (mb)\bar{u} = \bar{u}$. In particular $(\Id-\varphi (mb))u \in L$.

In the above, if we let $b$ be one of the canonical generators of $\Z^d$, then we see that we can control each coordinate and obtain an upper bound on $\rho(u,v)$ as $d\lambda^{d\norm{v}}$, since $\modulus{\det(\Id-\varphi(v))} \leq (1+\lambda_v)^d$,  and $\lambda_v \leq \lambda^{\norm{v}}$, with $\lambda_v$ and $\lambda$ as in step 2. Hence, in particular, $\rho(u,v) \leq d(1+\lambda^{\modulus{(u,v)}_\Gamma})^d$.

\vspace{2mm}
\noindent \underline{\textsc{Step} 4:} Estimating $\CLF_\Gamma(n)$.
\vspace{2mm}

\noindent By Theorem \ref{thm:group extensions}, we use the above bounds on $\RCLF_{\Z^d}^\Gamma(n)$, $\TCLF_{\Z^d}(n;\varphi(v))$ and $\rho(u,v)$ to obtain an upper bound on the conjugacy length function of $\Gamma$ as
		\[\CLF_\Gamma(n) \leq \max \left\{ B_2n , \rho_n + \invdist_{\Z^d}^\Gamma\big(2(1+\lambda^n){\delta_{\Z^d}^\Gamma(n+\rho_n)}\big)\right\}\]
where $\rho_n=n$ if $k=1$ or $\rho_n=d(1+\lambda^{n})^d$ if $k>1$. Using the estimates on the distortion functions obtained in Lemmas \ref{lem:lower bound - distortion in Z^d rtimes Z^k} and \ref{lem:upper bound - distortion in Z^d rtimes Z^r} we obtain a linear upper bound on the conjugacy length function for $\Gamma$ when $k=1$ and exponential otherwise.
\end{proof}

The exponential bound in Theorem \ref{thm:generalised SOL} arises because of the way the projection of $Z_\Gamma (u,v)$ onto the $\Z^k$--component lies inside $\Z^k$. In particular, one may ask if the exponential upper bound is sharp:

\begin{question}
Can one find a pair of conjugate elements in $\Gamma$ whose shortest conjugator is exponential in the sum of the lengths of the two given elements?
\end{question}

We will now discuss some applications of Theorem \ref{thm:generalised SOL} to other situations, firstly to the fundamental group of prime $3$--manifolds and secondly to Hilbert modular groups.

Let $M$ be a prime $3$--manifold with fundamental group $G$. Recent work of Behrstock and Dru\c{t}u \cite[\S 7.2]{BD11} has shown that, when $M$ is non-geometric, there exists a positive constant $K$ such that two elements $u,v$ of $G$ are conjugate only if there is a conjugator whose length is bounded above by $K(\modulus{u}+\modulus{v})^2$. Theorem \ref{thm:generalised SOL} in the case when $d=2$ and $k=1$ deals with the solmanifold case, while work of Ji, Ogle and Ramsey \cite[\S 2.1]{JOR10} gives a quadratic upper bound for nilmanifolds. These, together with the result of Behrstock and Dru\c{t}u, give the following:

\begin{thm}
Let $M$ be a prime $3$--manifold with fundamental group $G$. For each word metric on $G$ there exists a positive constant $K$ such that two elements $u,v$ are conjugate in $G$ if and only if there exists $g\in G$ such that $ug=gv$ and \[\modulus{g} \leq K(\modulus{u}+\modulus{v})^2\textrm{.}\]
\end{thm}

We now apply Theorem \ref{thm:generalised SOL} to the conjugacy of elements in parabolic subgroups of Hilbert modular groups. Such subgroups are isomorphic to a semidirect product $\Z^n \rtimes_\varphi \Z^{n-1}$, where $\varphi$ depends on the choice of Hilbert modular group and the boundary point determining the parabolic subgroup (see for example either \cite{Geer88} or \cite{Hirz73}). Because there is a finite number of cusps (see for example \cite{Shim63} or \cite{Geer88}), for each Hilbert modular group there are only finitely many $\varphi$ to choose from. Hence, by Theorem \ref{thm:generalised SOL}, any two elements in a parabolic subgroup are conjugate if and only if there exists a conjugator whose size is bounded exponentially in the sum of the sizes of the two given elements. More specifically:

\begin{cor}\label{cor:parabolic in Hilbert modular group}
Let $\Gamma = \SLOK$ be the Hilbert modular group corresponding to a finite, totally real field extension $K$ over $\Q$ of finite degree. There exists a positive constant $A$, depending only on $\Gamma$, such that a pair of elements $u,v$ in the same parabolic subgroup of $\Gamma$ are conjugate in $\Gamma$ if and only if there exists a conjugator $\gamma \in \Gamma$ such that
\[\modulus{\gamma} \leq A^{\modulus{u}+\modulus{v}} \textrm{.}\]
Furthermore, if $u,v$ are actually unipotent elements in $\Gamma$, then this upper bound is linear.
\end{cor}

\begin{proof}
Since $u,v$ are in the same parabolic subgroup of $\Gamma$ then Theorem \ref{thm:generalised SOL} gives the first conclusion. The second conclusion, for unipotent elements, follows from the linear upper bound on the restricted conjugacy length function in the proof of Theorem \ref{thm:generalised SOL}.
\end{proof}

In \cite{Sale12liegroup} we look at conjugacy in semisimple real Lie groups. We are able to show that certain pairs of elements, including most unipotent elements, enjoy a linear conjugacy relationship. A consequence of this is an extension of the result for unipotent elements in Corollary \ref{cor:parabolic in Hilbert modular group} to lattices in other semisimple Lie groups.

\bibliography{../../../../AWS_bibli/bibliography}{}
\bibliographystyle{alpha}

\end{document}